\documentclass[10pt]{amsart}
\usepackage{amsmath,amsthm,amsfonts,amssymb,amscd}
\usepackage[dvips]{graphicx}
\usepackage{psfrag}

\theoremstyle{plain}
\newtheorem{thm}{Theorem}[section]

\newtheorem{lemma}[thm]{Lemma}

\newtheorem{maintheorem}{Theorem}

\title[Pseudoconnections and Ricci flow]{Pseudoconnections and Ricci flow}

\author[C. Morales]{C. Morales}



\address{
C. Morales\\
Instituto de Matematica\\ Universidade Federal do Rio de Janeiro\\
P. O. Box 68530, 21945-970\\
Rio de Janeiro, Brazil.}

\begin{document}

\begin{abstract}
In this note we explain how a flow in the space of Riemmanian metrics (including Ricci's \cite{mt}) induces
one in the space of pseudoconnections.
\end{abstract}

\maketitle

Let $M$ be a differentiable manifold.
Denote by $C^\infty(M)$ the ring of all $C^\infty$ real valued maps defined in $M$. Given a vector bundle $\xi$ over $M$ we denote by $\Omega^0(\xi)$ the $C^\infty(M)$-module
of $C^\infty$ sections of $\xi$.
We write $\mathcal{X}$ instead of $\Omega^0(TM)$.
A {\em tensor field of order $p$ of $M$}
is a $C^\infty(M)$-linear map
$T:\mathcal{X}^p\to C^\infty(M)$, where $\mathcal{X}^p$ is the product of $p$ copies of $\mathcal{X}$.
We say that $T$ is {\em nondegenerated} if $T(X,\cdots,X)>0$ for all
$X\neq 0$, and {\em symmetric} if
$T(\cdots, X,\cdots, Y,\cdots)=T(\cdots, Y,\cdots, X,\cdots)$ for all $X,Y\in
\mathcal{X}$.
Denote by $Sym_p(M)$ the set of all symmetric
tensor fields of order $p$ of $M$.
A {\em Riemannian metric} of $M$ is a nondegenerated element of $Sym_2(M)$.
We denote by $Riem(M)$ the space of Riemannian metrics of $M$.

A {\em pseudoconnection} of a vector bundle $\xi$ over $M$ is an $I\!\! R$-bilinear map
$Q:\mathcal{X}\times \Omega^0(\xi)\to\Omega^0(\xi)$ for which there is a homomorphism of $C^\infty(M)$-modules
$P(Q):\Omega^0(\xi)\to \Omega^0(\xi)$, called {\em principal homomorphism}, such that
$$
Q_{fX}s=fQ_Xs \quad \mbox{ and }
\quad Q_X(fs)=X(f)P(Q)(s)+fQ_Xs
$$
for all $(X,s)\in \mathcal{X}\times \Omega^0(\xi)$
and all $f\in C^\infty(M)$ (e.g. \cite{s}). A connection of $\xi$ is a
pseudoconnection $\nabla$ for which $P(\nabla)$ is the identity homomorphism.
By a pseudoconnection (resp. connection) of $M$ we mean a pseudoconnection
(resp. connection) of its tangent bundle $TM$. A pseudoconnection $Q$ of $M$ is {\em symmetric} if $Q_XY-Q_YX=P(Q)([X,Y])$ for all $X,Y\in \mathcal{X}$, where $[X,Y]$ denotes the Lie bracket.

It is possible to obtain symmetric pseudoconnections of $M$ from
the following lemma.

\begin{lemma}
\label{l1}
For every map $R:Riem(M)\to Sym_2(M)$ and every
$g\in Riem(M)$ there is a symmetric pseudoconnection
$R^g$ of $M$ satisfying
$$
R(g)(X,Y)=g(P(R^g)(X),Y),
\,\,\,\,\,\,\,\forall X,Y\in \mathcal{X}.
$$
\end{lemma}

\begin{proof}
Define $R^g:\mathcal{X}\times \mathcal{X}\to\mathcal{X}$
implicitely by
$$
2g(R^g_XY,Z)=
X(R(g)(Y,Z))+Y(R(g)(Z,X))-Z(R(g)(X,Y))+
$$
$$
R(g)([X,Y],Z)
+R(g)([Z,X],Y)-R(g)([Y,Z],X),
\,\,\,\,\,\forall X,Y,Z\in \mathcal{X}.
$$
A straightforward computation shows not only that $R^g$ is well-defined
but also that it is a symmetric pseudoconnection of $M$.
\end{proof}

On the other hand, we know from the basic Riemannian geometry \cite{dcm} that
for every Riemannian metric $g$ there is a unique symmetric
connection $\nabla$ of $M$ which is {\em compatible} with $g$, namely,
$$
Xg(Y,Z)=g(\nabla_XY,Z)+g(Y,\nabla_XZ),
\,\,\,\,\,\,\,\,\,\,\forall X,Y,Z\in \mathcal{X}.
$$
This is the so-called {\em Levi-Civita connection} of $g$.

The objetive of this paper is to relate pseudoconnections to {\em metric flows},
i.e., one-parameter families of
Riemannian metrics $g_t$ in $M$ solving the differential equation
\begin{equation}
\label{eq1}
\frac{\partial g_t}{\partial t}=R(g_t),
\end{equation}
for some $R:Riem(M)\to Sym_2(M)$.
The most famous example of a metric flow
on a manifold is the so-called {\em Ricci flow} in which $R(g)=Ric(g)$
stands for the Ricci tensor of $g$ (c.f. \cite{mt}).

Now we state our result.

\begin{maintheorem}
\label{thA}
To every solution $g_t$ of (\ref{eq1}) it corresponds
a one-parameter family of symmetric pseudoconnections $Q^t$
of $M$ such that the Levi-Civita connection $\nabla^t$
of $g_t$ evolves according to the following equation
\begin{equation}
\label{eq3}
\frac{\partial\nabla^t}{\partial t}+P(Q^t)\circ\nabla^t=Q^t.
\end{equation}
\end{maintheorem}

\begin{proof}
Fix a solution $g_t$ of (\ref{eq1}) and their corresponding Levi-Civita connections $\nabla^t$. We know from
basic Riemannian geometry \cite{dcm} that $\nabla^{t}$
satisfies the so-called Kozul formula
$$
2g_t(\nabla^t_XY,Z)=X(g_t(Y,Z))+Y(g_t(Z,X))-Z(g_t(X,Y))+
$$
$$
g_t([X,Y],Z)+g_t([Z,X],Y)-g_t([Y,Z],X),
\,\,\,\,\,\,\forall X,Y,Z\in \mathcal{X}.
$$
Derivating it with respect to $t$
(as in the proof of Proposition 1.10 p. 21 in \cite{m})
and using (\ref{eq1}) we get
\begin{equation}
\label{eq2}
2g_t\left(\frac{\partial\nabla^t}{\partial t}(X,Y),Z\right)=-2R(g_t)(\nabla^t_XY,Z)+
X(R(g_t)(Y,Z))+
\end{equation}
$$
Y(R(g_t)(Z,X))-
Z(R(g_t)(X,Y))+
R(g_t)([X,Y],Z)
+R(g_t)([Z,X],Y)-
$$
$$
R(g_t)([Y,Z],X).
$$
On the other hand, applying Lemma \ref{l1} we obtain
$$
R(g_t)(\nabla^t_XY,Z)=g_t((P(R^{g_t})\circ\nabla^t)(X,Y),Z)
$$
and
$$
X(R(g_t)(Y,Z))+
Y(R(g_t)(Z,X))-
Z(R(g_t)(X,Y))+
R(g_t)([X,Y],Z)+
$$
$$
R(g_t)([Z,X],Y)-
R(g_t)([Y,Z],X)= 2g_t(R^{g_t}_XY,Z).
$$
So,
$$
g_t\left(\frac{\partial\nabla^t}{\partial t}(X,Y),Z\right)=
g_t\left(\left(-(P(R^{g_t})\circ\nabla^t)+R^{g_t}\right)(X,Y),Z\right),
\,\,\forall X,Y,Z\in \mathcal{X}.
$$
Since $g_t$ is a Riemmanian metric we conclude that
$$
\frac{\partial\nabla^t}{\partial t}+P(R^{g_t})\circ\nabla^t=R^{g_t},
\quad\forall t.
$$
Since $R^{g_t}$ is a symmetric pseudoconnection by Lemma \ref{l1}
we obtain (\ref{eq3}) by taking $Q^t=R^{g_t}$. The proof follows.
\end{proof}

\end{document}